\documentclass[12pt, oneside,reqno]{amsart}   	
\usepackage{geometry}                		
\allowdisplaybreaks
\usepackage{graphicx}				
\usepackage{amssymb}
\usepackage{amsmath}
\usepackage[all]{xy}
\usepackage{mathtools}
\usepackage{tikz-cd}
\usepackage{enumitem}
\usepackage{tikz-cd}
\usepackage{tabularx}
\usepackage{stmaryrd}

\newtheorem{thm}{Theorem}[section]
\newtheorem{lem}[thm]{Lemma}
\newtheorem{cor}[thm]{Corollary}
\newtheorem{prop}[thm]{Proposition}

\def\R{\ifmmode{\Bbb R}\else{$\Bbb R$}\fi}

\numberwithin{equation}{section}

\begin{document}

\title[Ancient and Eternal Solutions to Mean Curvature Flow]{Ancient and Eternal Solutions to Mean Curvature Flow from Minimal Surfaces}

\author{Alexander Mramor and Alec Payne}
\address{Department of Mathematics, University of California Irvine, Irvine, CA 92617}
\address{Courant Institute, New York University, New York City, NY 10012}
\email{mramora@uci.edu,ajp697@nyu.edu}

\date{}


\begin{abstract} We construct embedded ancient solutions to mean curvature flow related to certain classes of unstable minimal hypersurfaces in $\mathbb{R}^{n+1}$ for $n \geq 2$. These provide examples of mean convex yet nonconvex ancient solutions that are not solitons, meaning that they do not evolve by rigid motions or homotheties. Moreover, we construct embedded eternal solutions to mean curvature flow in $\mathbb{R}^{n+1}$ for $n \geq 2$. These eternal solutions are not solitons, are $O(n)\times O(1)$-invariant, and are mean convex yet nonconvex. They flow out of the catenoid and are the rotation of a profile curve which becomes infinitely far from the axis of rotation. As $t \to \infty$, the profile curves converge to a grim reaper for $n \geq 3$ and become flat for $n=2$. Concerning these eternal solutions, we also show they are asymptotically unique up to scale among the embedded $O(n)\times O(1)$-invariant, eternal solutions with uniformly bounded curvature and a sign on mean curvature.
\end{abstract}

\maketitle
\vspace{-.15in}
\section{Introduction}

Ancient solutions to mean curvature flow, i.e. solutions existing on the time interval $(-\infty, C]$, $- \infty < C \leq \infty$, play an important role in the singularity analysis of the flow as the natural blowup limits after rescaling about a singularity, making their study central in defining and understanding weak notions of the flow. A special type of ancient solution is the eternal solution, which is an ancient solution that exists for all time, i.e. it exists on the time interval $(-\infty, \infty)$. Eternal solutions are a subset of ancient solutions but are much more rigid and less is known about them. There are many known ancient solutions to mean curvature flow in Euclidean space, but only a small number of them are eternal, particularly if one excludes the translating solitons. Eternal solutions arise naturally as the blowup limits of Type II singularities, whereas non-eternal ancient solutions arise as the blowup limits of Type I singularities (see~\cite{Mant} for a description of Type I and Type II singularities). From a more analytic perspective, mean curvature flow is the natural analogue of the heat equation in the setting of submanifold geometry, and ancient solutions are the natural analogues of global solutions to elliptic equations, distinguishing ancient solutions in this sense. 
$\medskip$

It is useful then to have a wide variety of examples of ancient solutions to the flow to help understand the phenomena that could be realized by solutions to mean curvature flow. Ancient solutions can be split up among those which are solitons and those which are not solitons. By ``soliton'', we mean a solution to mean curvature flow which evolves by a combination of rigid motions and homotheties. And by ``non-soliton,'' we mean a solution which is not a soliton. There are far more soliton ancient solutions known than non-soliton ancient solutions. Ancient solutions may also be described as either convex or nonconvex. An ancient solution is called convex if every timeslice of the flow is a convex surface. Convexity is an important characteristic of an ancient solution as the ancient solutions which arise as blowup limits of mean convex mean curvature flow are convex~\cite{HS3, W2}.
$\medskip$

\vspace{-.03in}
Some examples of convex ancient solitons include the standard shrinking spheres and cylinders, the Abresch-Langer curves~\cite{AL}, and some of the rotating and shrinking solitons to curve shortening flow~\cite{Ha}. Examples of nonconvex ancient solitons include the Angenent torus~\cite{Ang}, desingularizations of the sphere and the Angenent torus \cite{KKM} and the sphere and the plane \cite{N1, N2, N3}, the high genus min-max constructions of Ketover \cite{Ket}, and many of the rotating and shrinking solitons to curve shortening flow \cite{Ha}. Among the convex eternal solitons are the grim reaper, the bowl soliton~\cite{AltWu}, the strictly convex translating solitons lying in slabs~\cite{BLT1, HIMW}, and the non-rotationally symmetric entire translators of Wang \cite{W}. Finally, some examples of nonconvex eternal solitons include any non-flat minimal surface, the winglike translators or translating catenoids~\cite{CSS}, the periodic Scherk-like translators \cite{HMW}, translators associated to minimal surfaces like the Costa-Hoffman-Meeks surface \cite{DPN}, the Yin-Yang spiral in one dimension~\cite{Alt}, the purely rotating solitons in dimensions two and higher that are analogous to Yin-Yang spirals~\cite{HS}, the nonconvex translating tridents \cite{N, N5}, and the multitude of exotic immersed self shrinkers \cite{DK, DLM}. Despite the richness of the ancient solitons, relatively few examples of non-soliton ancient solutions are known. We will focus on non-soliton ancient solutions throughout this paper.

\begin{center}
\begin{table}[h]\caption{}\label{table 1}
\begin{tabular}{ |p{1.25cm}|p{6.3cm}|p{6.35cm}|  }
 \hline
 \multicolumn{3}{|c|}{\textbf{Non-soliton Ancient Solutions in Euclidean Space}} \\
 \hline
 & \hfil (Strictly) Convex & \hfil Nonconvex\\
 \hline
 \vfil  \vspace{.55in} Eternal   & \begin{itemize}[leftmargin=.4cm]
     \item Assuming strict convexity and that the curvature attains its maximum at a point in spacetime, non-soliton convex eternal solutions do not exist~\cite{Ham}. In \cite{W2}, White conjectured that any nonflat convex eternal solution is a translating soliton.
 \end{itemize}    &\begin{itemize}[leftmargin=.4cm]
    \item The ancient sine curve in $\mathbb{R}^2$ \cite{NIW}, also known as the hairclip, as well as the closely related truncated versions of the ancient sine curve \cite{You}
    \item The examples of Corollary \ref{nonconvex eternal} in $\mathbb{R}^{n+1}$, $n \geq 2$
 \end{itemize}  \\
 \hline
 \vfil \vspace{.3in}Non-eternal & \begin{itemize}[leftmargin=.4cm]
     \item The Angenent oval~\cite{Ang}
     \item The ancient ovals of White \cite{W2} and Haslhofer-Hershkovits \cite{HH1}
     \item The ``ancient pancakes'' in a slab due to Bourni-Langford-Tinaglia~\cite{BLT} and Wang \cite{W}
 \end{itemize}    &\begin{itemize}[leftmargin=.4cm]
     \item Glued grim reapers forming immersed ``ancient trombones'' in $\mathbb{R}^2$~\cite{AngYou}
     \item The embedded examples of Corollary \ref{nonconvex ancient solutions} in $\mathbb{R}^{n+1}$, $n \geq 2$
 \end{itemize} \\
 \hline
\end{tabular}
\end{table}
\end{center}

In Table \ref{table 1}, we give a survey of the known non-soliton ancient solutions according to whether they are eternal or non-eternal and whether they are convex or nonconvex. We restrict ourselves to the codimension one case in Euclidean $\mathbb{R}^{n+1}$. Note that any ancient solution can be extended to higher dimensions by simply multiplying by isometric factors of $\mathbb{R}^k$, although we leave out these possibilities. In fact, any weakly convex ancient solution can be taken to be a strictly convex ancient solution multiplied by $\mathbb{R}^k$~\cite{HS}.
$\medskip$

Our first theorem, inspired in part by the interesting work of Choi-Mantoulidis \cite{ChM}, concerns the existence of ancient solutions flowing out of certain unstable minimal surfaces in $\mathbb{R}^{n+1}$ for $n \geq 2$. These provide examples for the bottom right box in Table \ref{table 1}. In the following theorem, we will need a technical assumption, which is satisfied by a large class of minimal surfaces. We say that a surface $M$ satisfies the $\textit{uniform tubular neighborhood assumption}$ if there exists a tubular neighborhood of uniform width such that the boundary of this tubular neighborhood is smooth and embedded. That is, there exists some $\delta>0$ such that $M+\delta\nu$ is smooth and embedded, where $\nu$ is a unit normal. In other words, $M$ satisfies the uniform tubular neighborhood assumption if it does not asymptotically approach itself. 

\begin{thm}\label{first theorem} For $n \geq 2$, let $M^n \subset \mathbb{R}^{n+1}$ be an unstable\footnote{In line with \cite{FCS}, we mean than on some bounded domain $D \subset M$, $\lambda_1(D) < 0$ for the Jacobi operator.} 2-sided properly embedded minimal surface which has uniformly bounded curvature, satisfies the uniform tubular neighborhood assumption (as defined above), and satisfies one of the following options: 
\begin{enumerate}
\item it is asymptotically flat, or
\item it is periodic with compact fundamental domain, or 
\item it is periodic and asymptotically flat in its fundamental domain.
\end{enumerate}

Then, there exist two distinct ancient solutions $M^1_t$ and $M^2_t$ to mean curvature flow such that $M^1_t$ and $M^2_t$ smoothly and uniformly converge to $M$ from opposite sides of $M$ as $t \to -\infty$. These ancient solutions are embedded and have a sign on mean curvature yet are nonconvex and are not solitons. 
\end{thm} 

In the proof of Theorem \ref{first theorem}, the assumptions of asymptotic flatness or periodicity plays a subtle role in the asymptotics of the ancient solution, but note that these assumptions include a wide class of minimal surfaces. Some nontrivial examples of minimal surfaces satisfying the assumptions of this theorem are catenoids and the Costa-Hoffman-Meeks surfaces. Note also that minimal surfaces with infinitely many ends such as the Riemann examples are covered by item (3).  
$\medskip$
 
The ancient solutions constructed in Theorem \ref{first theorem} seem to be the first known instances of nonconvex embedded ancient solutions in $\mathbb{R}^{n+1}$ which are not solitons. To the authors' knowledge the only previously constructed nonconvex non-eternal ancient solutions that are not solitons are the immersed curves in $\mathbb{R}^2$ of Angenent and You \cite{AngYou}. In fact, to the authors' knowledge, all previously known mean convex non-eternal non-soliton ancient solutions to mean curvature flow have been convex. Since the ancient solutions of Theorem \ref{first theorem} have a sign on mean curvature, i.e. they are mean convex with the correctly chosen normal field, we have the following corollary.

\begin{cor}\label{nonconvex ancient solutions} There exist mean convex, yet nonconvex, non-soliton ancient solutions to mean curvature flow in $\mathbb{R}^{n+1}$, $n \geq 2$.  
\end{cor}

Our second theorem concerns the existence of an eternal solution flowing out of the catenoid in $\mathbb{R}^{n+1}$ for each $n \geq 2$. This provides an example for the top right box in Table \ref{table 1}. In the following theorem, let $M^1$ be a catenoid in $\mathbb{R}^{n+1}$ for $n \geq 2$. Center $M^1$ so that it is rotationally symmetric about an axis passing through the origin. The catenoid splits $\mathbb{R}^{n+1}$ into two connected components, the ``inside'' and the ``outside''. Let $\nu$ be the unit normal on the neck of the catenoid such that $\nu$ points away from the origin. Then, let the ``outside'' of the catenoid be the connected component that $\nu$ points into.

\begin{thm}[The Reapernoid]\label{second theorem}
For each $n \geq 2$, there exists a mean convex\footnote{This is with respect to the normal on $M^1_t$ compatible, as $t \to -\infty$, with the normal $\nu$ to the catenoid $M^1$ (see the discussion above the theorem).}  $O(n) \times O(1)$-invariant eternal solution $M^1_t$ to mean curvature flow in $\mathbb{R}^{n+1}$ with uniformly bounded curvature such that for each $t$, $M^1_t$ is a subset of the outside of the catenoid $M^1$ (as defined above) and $M^1_t$ converges smoothly and uniformly to $M^1$ as $t \to -\infty$.

As $t \to \infty$, $M^1_t$ becomes infinitely far from its axis of rotation. For $n \geq 3$, the profile curve of $M^1_t$ will converge as $t \to \infty$ to a grim reaper of the same width as $M^1$. For $n=2$, the pointed limit of the profile curve of $M^1_t$ is a line and the curvature of $M^1_t$ approaches zero as $t \to \infty$.
\end{thm}
\begin{figure}
\centering
\includegraphics[scale = .6]{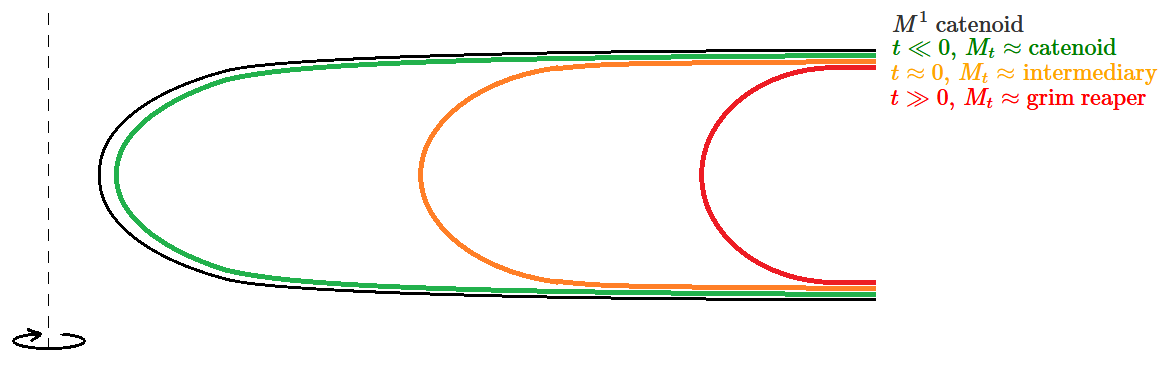}
\caption{A sketch of the regimes of the profile curves of the higher-dimensional reapernoid, the eternal solution of Theorem \ref{second theorem}. For $t \ll 0$, the eternal solution has a profile curve close to that of the catenoid, and for $t \gg 0$, it has a profile curve close to that of a grim reaper.}
\end{figure}

By a theorem of Richard Hamilton~\cite{Ham}, a strictly convex eternal solution which achieves its spacetime maximum of curvature must be a translating soliton. The eternal solution constructed in Theorem \ref{second theorem}, which will be referred to as the reapernoid as a reminder of its asymptotics, seems to be the first known instance of a non-soliton eternal solutions in $\mathbb{R}^{n+1}$, $n\geq 2$, which does not split off a line. To the authors' knowledge, the only previously known non-soliton eternal solutions to mean curvature flow are curves in $\mathbb{R}^2$ or their isometric products with $\mathbb{R}^k$. Thus, we find the following corollary.

\begin{cor}\label{nonconvex eternal}
For $n \geq 2$, there exists an eternal solution to mean curvature flow in $\mathbb{R}^{n+1}$ that is not a soliton and does not split off a line.
\end{cor}

Our final result concerns a partial uniqueness statement for the reapernoid eternal solution of Theorem \ref{second theorem}. Its proof is not particularly difficult but naturally leads into a host of further questions which we discuss in the concluding remarks:

\begin{thm}\label{uniqueness} Suppose $M_t$ is a connected embedded nonflat eternal solution to mean curvature flow in $\mathbb{R}^{n+1}$ which
\begin{enumerate}
\item is $O(n) \times O(1)$-invariant,
\item has a sign on mean curvature, and
\item has uniformly bounded curvature for all time.
\end{enumerate}

Then $M_t$ is either the catenoid itself or it has the asymptotics of the eternal solution of Theorem \ref{second theorem} up to scale. That is, $M_t$ converges to the catenoid from the outside as $t \to -\infty$, and as $t \to \infty$, the profile curve of $M_t$ converges to a grim reaper for $n\geq 3$ or becomes flat for $n=2$.
\end{thm}

In other words, this shows that the reapernoid, the eternal solution of Theorem \ref{second theorem}, is asymptotically unique among $O(n)\times O(1)$-invariant eternal flows with uniformly bounded curvature and a sign on mean curvature. Note that the translating bowl soliton is excluded from the above conditions because it is $O(n)$-invariant but not $O(n)\times O(1)$-invariant. Also, in Theorem \ref{uniqueness}, we do not prove that the grim reapers found in the limit as $t \to \infty$ are necessarily of the same width as the catenoid, as in Theorem \ref{second theorem}.
$\medskip$

$\textbf{Acknowledgements:}$  The authors would like to thank Kyeongsu Choi and Christos Mantoulidis for responding to comments and writing an interesting and inspiring paper~\cite{ChM} where they construct ancient flows out of compact minimal surfaces in non-Euclidean ambient spaces, using different methods than those of this paper. The authors would also like to thank Mat Langford and Shengwen Wang for their comments and suggestions. Finally, the authors thank their advisors, Richard Schoen and Bruce Kleiner, respectively, for their support and advice.

\section{Preliminaries, Old and New}
In this section we collect some standard and nonstandard facts on mean curvature flow and minimal surfaces, which will be used later to streamline the proofs. Let $M$ be an $n$-dimensional orientable manifold and let $F: M \to \mathbb{R}^{n+1}$ be an embedding of $M$ realizing it as a smooth closed $2$-sided hypersurface of Euclidean space, which by abuse of notation we also refer to as $M$. Then the mean curvature flow $M_t$ is given by the image of $\hat{F}: M \times [0,T) \to \mathbb{R}^{n+1}$ satisfying
\begin{equation}\label{MCF equation}
\frac{d\hat{F}}{dt} = H \nu, \text{ } \hat{F}(M, 0) = F(M)
\end{equation}
where $\nu$ is a unit normal and $H$ is the mean curvature. It turns out that (\ref{MCF equation}) is a nonlinear heat-type equation, since for $g$ the induced metric on $M$,
\begin{equation}
 \Delta_g F = g^{ij}(\frac{\partial^2 F}{\partial x^i \partial x^j} - \Gamma_{ij}^k \frac{\partial F}{\partial x^k}) = g^{ij} h_{ij} \nu = H\nu
\end{equation}
That the left-hand side is the Laplacian motivates the assertion that the mean curvature flow is the natural analogue of the heat equation in submanifold geometry. One can easily see that the mean curvature flow equation (\ref{MCF equation}) is degenerate. Despite this, solutions to (\ref{MCF equation}) always exist for short time and are unique provided that the initial data has bounded second fundamental form. There are several ways to deduce this by relating (\ref{MCF equation}) to a nondegenerate parabolic PDE. Solutions to mean curvature flow satisfy many properties that solutions to heat equations do, such as the maximum principle and smoothing estimates. One important consequence of the maximum principle is the comparison principle (also known as the avoidance principle), which says that two initially disjoint hypersurfaces will remain disjoint over the flow. More generally, for two noncompact hypersurfaces with uniformly bounded geometry, if the flows $M^1_t$ and $M_2^t$ are initially distance $\delta > 0$ apart they remain so under the flow (see, for instance, Remark 2.2.8 of \cite{Mant}). In fact, in the cases of interest to us in this paper, such a comparison principle for two noncompact hypersurfaces can be proven independently. We are interested in applying the comparison principle between noncompact hypersurfaces of uniformly bounded geometry that are either asymptotically flat or periodic (or a combination of both). Indeed, between periodic surfaces with compact fundamental domain, the standard comparison principle generalizes immediately. And for asymptotically flat surfaces, pseudolocality (see Chen-Yin \cite{CY}) keeps the ends arbitrarily stationary, meaning that the separated flows must have an interior minimum of distance if they approach each other.
$\medskip$

Now we give some preliminary facts more specific to the proofs below:
 
\subsection{Preliminaries for the proof of Theorem 1.1}
$\medskip$

Singularities along the flow can only occur at points and times where the norm of the second fundamental form $A$ blows up. Hence, to rule out singularities, we need curvature estimates. Our method to find curvature estimates in this section is the Brakke-White regularity theorem \cite{B}. The version of this theorem stated below is due to Brian White \cite{W1}: 
\begin{thm}[Brakke, White]\label{Brakke} There are numbers $\epsilon_0 = \epsilon_0(n) > 0$ and $C = C(n) < \infty$ with the following property. If $\mathcal{M}$ is a smooth mean curvature flow starting from a hypersurface $M$ in an open subset $U$ of the spacetime $\mathbb{R}^{n+1} \times \mathbb{R}$ and if the Gaussian density ratios $\Theta(M_t, X, r)$ are bounded above by $1 + \epsilon_0$ for $0 < r< \rho(X,U)$, then each spacetime point $X = (x,t)$ of $\mathcal{M}$ is smooth and satisfies:
 \begin{equation}
 |A|^2 \leq \frac{C}{\rho(X, U)}
 \end{equation}
 where $\rho(X,U)$ is the infimum of $||X - Y||$ among all spacetime points $Y \in U^c$.
 \end{thm} 
 
In the above theorem we recall that the Gaussian density ratio $\Theta(M_t, X, r)$ is given by
 \begin{equation}
 \Theta(M_t, X, r) = \int_{y \in M_{t - r^2}} \frac{1}{(4\pi r^2)^{n/2}} e^\frac{-|y - X|^2}{4r^2} d \mathcal{H}^n(y)
 \end{equation} 
By Huisken's monotonicity formula~\cite{H}, this quantity is monotone nondecreasing in $r$. So, to get curvature bounds via the regularity theorem, we only need to sufficiently bound a range $[r_1, r_2] \subset (0, \infty)$ of the densities in an open set $U$ for some time interval $[t_1, t_2]$ with $r_1^2 < t_2 - t_1$. 
$\medskip$

The curvature estimates in the proof of Theorem \ref{first theorem} will depend on Proposition \ref{smush}, which is proven with the Brakke regularity theorem. Note that the $C^2$ bounds on $f$ below are precisely curvature estimates for $M_t$: 

 \begin{prop}\label{smush} Let $M_t$ be a flow for $t \in [0,T)$, $T>1$. Let $N^1$ and $N^2$ be smooth properly embedded hypersurfaces that are disjoint and have $|A|^2$ uniformly bounded by $C< \infty$. Suppose that
 \begin{enumerate}
  \item $M_0$, $N^1$, and $N^2$ all satisfy one of assumptions (1)-(3) of Theorem \ref{first theorem}, i.e. they are all either asymptotically flat, periodic with compact fundamental domain, or are periodic and asymptotically flat in their domain,
  \item $M_t$ lies between hypersurfaces $N^1$ and $N^2$ for $t \in [0,T)$,
 \item $M_0$ is a graph of a function $f$ over $N^1$ with $||f||_{C^2} < \rho$, and
 \item the distance between $N^1$ and $N^2$ is uniformly bounded by $\eta > 0$,
 \end{enumerate}
 Then there is $\overline{\eta} > 0$ and $D \gg 0$ depending on $\rho$ and $C$ but not $T$ such that if $\eta < \overline{\eta}$, the flow of $M_t$ will be a graph of a function $f_t$ over $N^1$ with $||f_t||_{C^2} < D$ for $t \in [0,T)$. Thus, $M_t$ will exist with uniformly bounded curvature as long as it lies between $N^1$ and $N^2$. 
 \end{prop}
 \begin{proof}
 We first note by continuity of the flow there is some small $s$ (depending on the bound $C$) so if condition (2) above is satisfied it will remain so for $2\rho$ on $[0,s]$ for some function $f_t$ defined on $[0,s]$ with $f_0 = f$. 
$\medskip$

To deal with later times we will use the Brakke regularity theorem. More precisely, from the $C^2$ bound $2\rho$ on $f_t$ for $t \in [0,s]$ and the $C^0$ bounds that come from choosing $\eta$ small enough, we find that we can obtain $C^1$ bounds on $f_t$ which approach $0$ as $\eta \to 0$. Note that the $C^1$ bounds on $f_t$ depend only on $\eta$, the $C^2$ bound, and $C$. Then, choosing $\eta$ small enough, we find that the area of $M_t$ as a graph over some ball in $N^1$ is an arbitrarily small multiple of the area of that ball in $N^1$. We may then apply the Brakke regularity theorem over uniformly small balls to find that at time $s$, $||f_s||_{C^2} < D$ and, in particular using \cite{EH2}, one may continue the smooth flow. Replacing $\rho$ above with $D$ and replacing $s$ with the corresponding doubling time $s'$, we get $||f_t||_{C^2}<2D$ for times $t \in [s, s + s']$. Note that the doubling time depends only on $C$ and $D$ and not on $f_t$ itself. This follows from assumption (1) above since $M_t$ is initially periodic or asymptotically flat and will remain so for as long as it exists. This means that $M_t$ must attain an interior maximum of curvature and thus its doubling time depends only on $C$ and $D$. Then, choose $\eta$ small enough to find small enough $C^1$ bounds on $f_t$ for $t \in [s,s+s']$ to apply the Brakke regularity theorem over the same uniformly small balls as before. This gives that $||f_t||_{C^2} < D$ for $t \in [s, s+s']$. Then, we may iterate the argument using $||f_{s+s'}||_{C^2}<D$ while keeping $\eta$ the same as long as the flow exists between $N^1$ and $N^2$.
\end{proof}

\subsection{Preliminaries for the proof of Theorem 1.3}\label{eternal preliminaries}
$\medskip$

The curvature estimates in the proof of Theorem \ref{second theorem} need a different approach than those of Theorem \ref{first theorem}. The following result of Ecker-Huisken~\cite{EH2} (cf. Corollary 3.2 (ii)) will be used in the proof of Theorem \ref{second theorem}:

\begin{thm}[Ecker-Huisken~\cite{EH2}]\label{EH estimate} Let $\omega$ be a fixed vector in $\mathbb{R}^{n+1}$ and let $R>0$ and $0\leq \theta <1$. Let $B(y_0, R)$ be a ball in a hyperplane orthogonal to $\omega$. Suppose that a mean curvature flow $M_t$ may be written as a compact graph over $B(y_0, R)$ for time $t \in [0,T]$. Then, for $t\in [0,T]$,
\begin{equation}
    |A|^2_{B(y_0, \theta R)}(t) \leq C(n) (1-\theta^2)^{-2}\big(\frac{1}{R^2}+\frac{1}{t}\big) \sup_{B(y_0, R)\times [0,t]} v^{4}
\end{equation}
where $v = \langle \nu_{M_t}, \omega \rangle^{-1}$ and $\nu_{M_t}$ is a unit normal to $M_t$. 
\end{thm}

In the proof of Theorem \ref{second theorem}, we will be working directly with the catenoid, so we will list important facts about the geometry of the catenoid and rotationally symmetric flows in general. 
$\medskip$

Let $M^1$ be a catenoid in $\mathbb{R}^{n+1}$ of fixed radius centered around the origin. Arrange the catenoid so that it may be represented as the rotation of a positive graph $y=w(x)$ around the $x$-axis. Arrange and scale the catenoid so that it is symmetric about the $y$-axis and $w(0)=1$. That is, the minimum point of the graph $u$ is located at $w(0)=1$. The catenoid is the unique nonflat minimal surface, up to scaling, that arises as a surface of rotation in this way. The catenoids of different scales are given by $R\,w(\frac{x}{R})$ for $R>0$. For $n\geq 3$, the catenoid of radius $1$ has a finite half-width $W_n$ given by
\begin{equation}
    W_n = \frac{1}{2}\int_{s=1}^{\infty}\frac{1}{\sqrt{s^{2(n-1)}-1}} ds
\end{equation}
For $n=2$, the catenoid has an infinite half-width, so let $W_2=\infty$. Note that others have defined $W_n$ as twice what we have defined it here, but we will have need of this normalization later on.
$\medskip$

We define the outward unit normal to $M^1$ to be the unit normal that points away from the axis of rotation. Also, $M^1$ separates $\mathbb{R}^{n+1}$ into two connected components: the ``inside'' and the ``outside'' of the catenoid. The outside of the catenoid is defined to be the connected component that outward unit normal to $M^1$ points into. And the ``inside'' is the other component.
$\medskip$

The mean curvature flow of a surface given by the rotation of a graph is particularly simple; it is equivalent to the flow of the graph satisfying the following equation.

\begin{equation}\label{MCF rotationally symmetric}
\frac{\partial u}{\partial t} = \frac{u_{xx}}{1+u_x^2} - \frac{n-1}{u}
\end{equation}

Indeed, if a surface is initially given by the rotation of a graph, then its flow will be given by the rotation of a graph for as long as it exists. Abstractly, this follows from the Sturmian theory of such flows, developed in \cite{AAG}. For the graph $w$ of the profile curve of $M^1$, the right hand side of (\ref{MCF rotationally symmetric}) vanishes, which is consistent with the fact that $M^1$ is a minimal surface. An important observation is also that, as $u$ increases, the flow is better and better approximated by the curve shortening flow of the graph of u. Indeed, (\ref{MCF rotationally symmetric}) without the second term on the right-hand side is just the curve-shortening flow of a graph. 
$\medskip$

Finally, for an embedded rotationally-symmetric surface, consider the points $p$ such that the unit normal $\nu(p)$ satisfies $\langle \nu(p), v\rangle>0$, where $v$ is a unit normal perpendicular to the axis of rotation and pointing away from the axis of rotation. Then, the mean curvature $H$ at $p$ is the following:
\begin{equation}\label{mean curvature of graph rotation}
H = k-\frac{n-1}{r}\cos{\theta}
\end{equation}
where $k$ is the curvature of the profile curve with respect to $\nu(p)$ and $\theta$ is the angle the tangent vector to the curve makes with the positively-oriented $x$-axis. Note that in the case that the curve is a graph with the appropriately chosen normal, every point will satisfy the condition $\langle \nu(p), v \rangle >0$. 

\section{Proof of Theorem 1.1} 

Throughout this section unless otherwise stated, $M\subset \mathbb{R}^n$ will denote a minimal surface as assumed in Theorem \ref{first theorem}.
$\medskip$

To construct our ancient solutions we will proceed as typical: we first construct ``old-but-not ancient'' solutions $(M_j)_t$ to the flow existing on $[0, j]$, $j \to \infty$, recentering the time coordinate to get flows existing on $[-j, 0]$, and take a limit of flows to obtain an ancient solution. To take the limit we need to have good enough estimates, and to show we get something nontrivial, we need to know that the limit flow is nonempty and not another ``known'' ancient solution to the flow, like the original minimal surface $M$.
$\medskip$

First we will prove the following general lemma which will help to construct the old-but-not-ancient flows and immediately give Lemma \ref{l1}.
$\medskip$

\begin{lem}\label{continuity lemma} Let $N$ be a hypersurface in $\mathbb{R}^{n+1}$ such that $N_t$ exists with uniformly bounded geometry for $t\in [0,T_0]$. For every $T<T_0$ and $\epsilon \ll 1$, there exists $\delta = \delta(N, \epsilon, T) \ll 1$ such that if $N^*$ is a hypersurface with $||N^* - N||_{C^2} \leq \delta$, then
\begin{equation}
||N^*_t - N_t||_{C^2} \leq \epsilon
\end{equation}
for all $t\in [0,T]$.
\end{lem}
\begin{proof}
Suppose not. Then, for some $T<T_0$ and $\epsilon\ll 1$, there is a sequence of $N^*_n$ such that $||N^*_n -N ||_{C^2} < \frac{1}{n}$ yet there exists $t_n \in (0,T]$ such that $||(N^*_n)_{t_n} - N_{t_n}||_{C^2} >\epsilon$. Note that since $N^*_n$ has uniformly bounded geometry independent of $n$, we have that the flow of $N^*_n$ will a priori exist on some short time interval independent of $n$. Then, the sequence $N^*_n$ converges to $N$ in $C^2$ so by continuity of the flow under perturbations of the initial conditions in $C^2$, we have that $(N^*_n)_t = N_t$ for $t\in [0,T]$. 
$\medskip$

Since $N$ has uniformly bounded geometry, we have that each $N^*_n$, for large $n$, has uniformly bounded geometry as well. However, since $[0,T]$ is a compact time interval, we may find a subsequence of $\{t_n\}$ converging to $t^*\in [0,T]$ such that $(N^*_n)_{t_n}$ subconverges to a limit surface $(N^*_{\infty})_{t^*}$ such that $||(N^*_{\infty})_{t^*} - N_{t^*}||_{C^2}>\epsilon$. However, by the above argument, $(N^*_n)_{t_n}$ must be converging to $N_{t^*}$. This is a contradiction.   
\end{proof}

To define one family of the old-but-not ancient solutions, $(M_j)_t$, we will consider perturbations $M_j = (M_j)_0$ of $M$ by small constant variations normal to $M$, for a given choice of unit normal $\nu$ (here we use the two-sided hypothesis). We can similarly define another family by switching the orientation of the normal throughout to obtain the second claimed family. 
$\medskip$

More precisely, let $M_{\delta}$ be given by $M_{\delta} = M + \delta \nu$. By the uniform tubular neighborhood assumption and bounded curvature, for all $\delta$ small, $M_{\delta}$ will be smooth and embedded. Note that the uniform tubular neighborhood assumption gives that there exists one $\delta$ such that $M_{\delta}$ is smooth and embedded, but with our assumptions, this implies that $M_{\delta}$ is smooth and embedded for all small $\delta$. For $j \to \infty$, $M_j$ will be defined to be $M_{\delta_j}$ for an appropriate choice of $\delta_j$. Apply Lemma \ref{continuity lemma} to the minimal surface $M$ with some fixed choice of $\epsilon$ and the time $T_j$. This gives some $\delta_j$ such that $M_j = M_{\delta_j}$ will exist for time $[0,T_j]$ and $M_j$ will be $\epsilon$-close to $M$ in $C^2$ for this same time. 
We will further refine this choice of $\delta_j$ after Lemma \ref{l3}. 
$\medskip$

Note that $M_{\delta}$ is (weakly) mean convex with respect to $\nu$ for all $\delta$ small enough such that $M_{\delta}$ is smooth and embedded. Indeed, suppose that it is not. Then, if there is a point $p \in M_{\delta}$ such that $H(p)<0$, for some small time $0<t\ll 1$, $||(M_{\delta})_t - M||_{C^0}<\delta$. This follows because $\nu$ points away from $M$, so if $M_{\delta}$ has negative mean curvature at $p$ with respect to $\nu$, the flow will force it to become closer to $M$ for some short time. However, by the comparison principle between noncompact hypersurfaces (see Section 2), $||(M_{\delta})_t - M||_{C^0}\geq\delta$ for all time $t$. This is a contradiction, so $M_{\delta}$ is mean convex. By Corollary 4.4 of Ecker-Huisken~\cite{EH2}, $(M_{\delta})_t$ will remain mean convex for as long as it exists. 
$\medskip$

In order to take a limit of the approximate solutions $(M_j)_t$, we need curvature bounds. These will come from Lemma \ref{l2}. We must also know that for large enough $j$, corresponding to small enough $\delta$, $(M_j)_t$ will exist for long enough. This is given by Lemma \ref{l1}. In order to extract a limit that is distinct from $M$, we need to know that there exists some $\epsilon_1 \ll 1$ such that for all $\delta$ small enough, $(M_{\delta})_t$ will flow to be distance $\epsilon_1$ away from $M$, at some point, after some amount of time. This will come from Lemma \ref{l3}.
$\medskip$

Now, we find the following immediate consequence of Proposition \ref{smush}:
\begin{lem}\label{l2} There is $\epsilon_1 > 0$ and $C,D \gg 0$ so that as long as the flow $(M_\delta)_t$ is a subset of the interior of the region $U$ between the minimal surface $M$ and the smooth embedded surface $M_{\epsilon_1}$ for $t \in [0, T]$, $T > 1$, then $|A|^2 < C$ on $(M_{\delta})_t$ and $(M_{\delta})_t$ is the graph of a function $f_t$ over $M_{\delta}$ with $||f_t||_{C^2} < D$ for $t \in [0, T]$. 
\end{lem}
The above lemma shows that the flows $(M_{\delta})_t$ will exist with uniformly bounded curvature as long as the flow is between $M$ and $M_{\epsilon_1}$. Now that we have set $\epsilon_1$, apply Lemma \ref{continuity lemma} to the minimal surface $M$, setting $N$ to be $M$, to find the following:
\begin{lem}\label{l1} Let $\epsilon_1$ be the constant obtained from Lemma \ref{l2}. Let $T_\delta$ be the first time $(M_{\delta})_{T_{\delta}} \cap M_{\epsilon_1} \neq \emptyset$. Then as $\delta \to 0$, $T_{\delta} \to \infty$.
\end{lem}

This lemma tells us that $(M_{\delta})_t$, for $\delta$ small enough, will exist for as long as it is a subset of the region $U$ between $M$ and $M_{\epsilon_1}$, with curvature bounds in this region independent of $j$ (from Lemma \ref{l2}). This allows us to extract an ancient limit flow from $(M_j)_t$. The only thing remaining is to ensure that the limit flow will be different from $M$. To do this, it suffices to show that for every sufficiently small $\delta$, $T_\delta \neq \infty$. That is, for all small enough $\delta$, $(M_{\delta})_t$ will eventually flow to intersect $M_{\epsilon_1}$, for $\epsilon_1$ chosen sufficiently small. This is the heart of the proof of Theorem \ref{first theorem} and where the assumption of instability is used. 

\begin{lem}\label{l3} Let $\epsilon_1$ be as chosen above. Let $T_{\delta}$ be the first time that $(M_{\delta})_{T_{\delta}} \cap M_{\epsilon_1} \neq \emptyset$. Then after possibly taking $\epsilon_1$ smaller, for all $\delta \ll 1$, $T_{\delta} \neq \infty$. 
\end{lem}

\begin{proof}
Suppose this is not the case for $\epsilon_1$. Then, there is some $0<\delta^*\ll 1$ so that $||(M_{\delta^*})_t - M||_{C^0} < \epsilon_1$ for all time $t$. 
By Lemma \ref{l2}, $(M_{\delta^*})_t$ will have uniform curvature bounds for all time. This means that $(M_{\delta^*})_t$ will exist for time $t \in [0,\infty)$ with uniformly bounded curvature and will remain between $M$ and $M_{\epsilon_1}$.  
$\medskip$

As proven above, $(M_{\delta^*})_t$ is mean convex, so $(M_{\delta^*})_t$ is moving monotonically away from $M$. If $M_{\delta^*}$ is minimal, then relabel it to $N$ and proceed to the next paragraph. Suppose that $M_{\delta^*}$ is not minimal. By the uniform curvature bounds on $(M_{\delta^*})_t$, we may pass to a limit along any subsequence of $t_i \to \infty$ to find that $(M_{\delta^*})_t$ smoothly converges to some smooth limit surface which we denote by $N$. In fact, $N$ is a minimal hypersurface. 
$\medskip$

Indeed, suppose $N$ is not minimal. Since $(M_{\delta^*})_t$ is mean convex, we know that $N$ must be mean convex as well. Then, there is $p \in N$ such that $H(p)> c>0$. We may find a smooth curve $p(t) \in (M_{\delta^*})_t$ such that $p(t) \to p$ and $p(t)$ is the spacetime track of a point converging to $p$. For $t$ large enough, $H(p(t))>\frac{c}{2}$. Since the flow moves monotonically by mean convexity, we have by integration that $d(p,p(t)) = \infty$. This contradicts the fact that $p(t)$ converges to $p$. Thus, $N$ is a smooth complete minimal hypersurface disjoint from $M$ yet is between $M$ and $M_{\epsilon_1}$.
$\medskip$

We then reset $\epsilon_1$ to be $\epsilon_1/2$. If the statement is true for $\epsilon_1/2$ then we are done; otherwise, we iterate the argument. Labeling $\epsilon_{1,k} = \epsilon_1/2^k$, we must have the conclusion either be true for some choice of $k$, or we obtain a sequence of distinct minimal surfaces approaching $M$ from one side.
$\medskip$

In the latter case, denote by $N_k$ each of the $N$ found above using $\epsilon_{1,k}$. By the curvature estimates coming from Lemma \ref{l2} they are graphical over $M$ for large enough $k$ since $\epsilon_{1,k} \to 0$.  Of course, since $\epsilon_{1,k} \to 0$, we have that $N_k$ converges from one side to $M$. This then gives rise to a positive solution to the Jacobi operator on $M$ as in \cite{Si}. By Theorem 1.1 in \cite{FCS} we must then have $\lambda_1(D) > 0$ on any bounded domain of $M$, contradicting the instability of $M$. 
\end{proof}

Hence for $\epsilon_1$ small enough, we know for every integer $j > 0$ there will be $\delta$ so that $T_{\delta} = j$. Let $\delta_j$ be such that $T_{\delta_j} = j$ and define $M_j$ to be $M_{\delta_j}$. We have that the flow $(M_{\delta_j})_t$ will exist for $[0,j]$ with curvature---and hence by Shi's estimates, the derivatives of curvature---bounded uniformly, independent of $j$. Recentering the time parameter by $-j$ for each of the $(M_j)_t$, we have that $(M_j)_t$ is defined for $[-j,0]$. By definition of $T_{\delta_j}$, $(M_j)_0$ will intersect $M_{\epsilon_1}$. Let $p_j \in (M_j)_0 \cap M_{\epsilon_1} \neq \emptyset$. Recenter each $(M_j)_0$ so that $p_j$ is taken to the origin. Then, take a subsequential limit in the smooth topology to find an ancient solution $M_t$.
$\medskip$

There is a catch though: it is conceivable that $M_t$ does not flow out of $M$ if the $p_j$ diverge to spatial infinity. For example, although this will soon be ruled out, if $M$ were a catenoid and the $p_j$ diverged, the limit of recenterings of $M$ would be a plane. 
$\medskip$

If $M$ is periodic with compact fundamental domain, the recenterings do not matter so we suppose that $M$ is asymptotically flat (the precise rate does not matter). In this case we will show that all the $p_j$ are contained in a compact set. Indeed, suppose not. Then the limit of the surface $M$ under recenterings will be flat and we will obtain in the limit an ancient solution flowing out of a plane $P$ distance $\epsilon_1$ from the origin but which, at $t = 0$, intersects the origin. 
$\medskip$

Going far enough back in time, there will be a time $T$ for which $M_T$ is at most distance $\epsilon_1/2$ from $P$. Considering an appropriate translate of $P$ by distance $3\epsilon_1/4$ though, we see by the comparison principle then that the flow will never be distance more than $3 \epsilon_1/4$ from $P$, giving a contradiction. 
$\medskip$

We get that the $p_j$ all lie in a bounded domain, so after recentering them all to the origin, the corresponding recenterings of $M$ result in $M$ moved by a finite translation. The ancient solution $M_t$ is not $M$, since it must be bounded away from $M$ at the origin at time $0$ by distance $\epsilon_1$. It is also certainly not a minimal surface because it flows out of $M$ but is distance $\epsilon_1$ from $M$ at $t = 0$, so is not stationary. 
$\medskip$

We note that if we took the unit normals with the opposite orientation from $\nu$, we would obtain a distinct ancient solution. These are distinct because they are approaching $M$ from opposite sides (since $M$ is $2$-sided). With respect to $\nu$, this other ancient solution has negative mean curvature.
$\medskip$

Finally, we can see that these ancient solutions are not solitons. If the ancient solution approaches the minimal surface $M$ from one side as $t \to -\infty$, this means that $M$ may not be translating as it must be slowing down to approach $M$. These ancient solutions may not be just rotating, as that would imply they would not be on just one side of $M$. Neither are these ancient solutions homothetically shrinking, as this would imply that they do not approach any surface as $t \to -\infty$. Finally, we have that combinations of these rigid motions are also impossible. Homothetic shrinking in combination with any other rigid motion is ruled out for the same reason. And these solutions cannot be translating and rotating at the same time, as both motions occur at some constant rates, which would imply that $M_t$ cannot be converging to any surface as $t \to -\infty$. This completes the proof of Theorem \ref{first theorem}.

\section{Proof of Theorem 1.3}

In this section, we will construct the eternal solution described in Theorem \ref{second theorem}, which exists in $\mathbb{R}^{n+1}$ for all $n\geq 2$. This eternal solution will be constructed using a catenoid, similar to what was done in Section 3.
$\medskip$

Recall the notation set in Section \ref{eternal preliminaries}. We let $M^1$ be a catenoid in $\mathbb{R}^{n+1}$, $n\geq 2$, normalized to have radius $1$. The catenoid $M^1$ has width $W_n$ (which is infinite for $n=2$) and is given by a graph $y=w(x)$ which is reflection symmetric about $x=0$, as described in the preliminaries. Moreover, we split up $\mathbb{R}^{n+1}$ into two components: the inside and the outside of the catenoid $M^1$. The outside of the catenoid is the component of $\mathbb{R}^{n+1}$ that the outward unit normal $\nu$ points into. The outward unit normal $\nu$ is the normal to $M^1$ that points away from the axis of rotation.
$\medskip$

By Theorem \ref{first theorem} there exists an ancient solution $M^1_t$ to mean curvature flow, such that $M^1_t$ uniformly converges to $M^1$ as $t \to -\infty$ and is a subset of the outside of $M^1$ for all time. This ancient solution is embedded and is not a soliton. Since $M^1_t$ smoothly converges to $M^1$ as $t \to -\infty$, we may equip $M^1_t$ with a unit normal that is compatible with the outward unit normal $\nu$ to $M^1$. This is the outward unit normal to $M^1_t$, and by Theorem \ref{first theorem}, $M^1_t$ is mean convex with respect to its outward unit normal. Recall that the approximate solutions used to construct $M^1_t$ are of the form $M^1_{\delta} = M^1+\delta \nu$ for $\delta \ll 1$. Since $M^1$ is $O(n) \times O(1)$-invariant and all $M^1_{\delta}$ have this symmetry as well, we get that $M^1_t$ is $O(n)\times O(1)$-invariant with respect to the same axes of symmetry of $M^1$. Similarly, $M^1_t$ may be represented as the rotation of a graph $u_t$, since $M^1_{\delta}$ can be, where $u_t$ is symmetric about $x=0$. We know that $M^1_t$ will remain a subset of the outside of $M^1$ for all time since $M^1_{\delta}$ all lie outside $M^1$ and the mean convexity of $M^1_t$ will force the flow to nest and thus avoid $M^1$ for as long as it exists. This means that $u_t>w$ for as long as $u_t$ exists. And by (\ref{mean curvature of graph rotation}) combined with mean convexity, we have that $u_t$ is convex for all time.
$\medskip$

The last useful property of $M^1_t$ is that it will remain asymptotic to $M^1$ for as long as it exists. The approximate solutions $M^1+\delta \nu$ are asymptotically flat and so by pseudolocality (see Chen-Yin \cite{CY}), they must remain arbitrarily close to $M^1+\delta \nu$ outside a large enough ball. This means that for as long as the flow $(M^1_{\delta})_t$ exists, it will remain asymptotic to $M^1 +\delta \nu$. This means that $M^1_t$, as the limit of these approximate solutions, will remain asymptotic to $M^1$ for as long as it exists.
$\medskip$

With this in hand, we will show that the ancient solution $M^1_t$ is in fact eternal. We will find its asymptotics later.

\begin{prop}\label{longtime existence}
For $n \geq 2$, let $M^1_t \subset \mathbb{R}^{n+1}$ be the ancient solution to mean curvature flow as described above. Then $M^1_t$ exists for all time, $t \in [0,\infty)$, and it is spatially asymptotic to $M^1$ for all time slices.
\end{prop}

\begin{proof}
In order to prove that $M^1_t$ exists for all time $t \in [0,\infty)$, we will show that for each $t$, there is a bound on $|A|^2$ for $M^1_t$.
$\medskip$

Let $u_t$ be the profile curve of $M^1_t$, and let $w$ be the profile curve of $M^1$, considered as graphs over the same axis with the same axes of symmetry. Since $M^1_t$ is mean convex and initially satisfies $u_t >w$, we have that $u_t >w$ for as long as it exists. 
$\medskip$

As mentioned in Section \ref{eternal preliminaries}, note that (\ref{MCF rotationally symmetric}) without the last term is merely the curve shortening flow of the graph $y=u(x)$. Since the second term is negative, a solution to (\ref{MCF rotationally symmetric}) is a subsolution to graphical curve-shortening flow. This means that a solution to (\ref{MCF rotationally symmetric}) starting at $y=u(x)$ will avoid the curve shortening flow of an appropriately chosen graph $y=v(x)$ such that $v(x)>u(x)+c$ for some $c>0$ independent of $x$. Indeed, since $u_t$ is a convex curve which converges uniformly to $w$ as $t \to -\infty$, we may place a grim reaper $\mathcal{G}$ of some half-width $\ell < W_n$ strictly above $u_0$ such that $\mathcal{G}$ avoids $u_0$. The grim reaper $\mathcal{G}$ will be a graph over $(-\ell, \ell)$ and its curve shortening flow will be given by $\mathcal{G}+ct$, where $c>0$ is a constant. We have that $M^1_t$ will remain asymptotic to $M^1$ for as long as it exists and $\mathcal{G}+ct$ will remain asymptotic to $x \pm \ell$. So, if $u_t$ ever intersects $\mathcal{G}+ct$, there will be a point $x_0 \in (-\ell, \ell)$ such that the difference in height between $u_t$ and $\mathcal{G}+ct$ will reach a strict local minimum at $x_0$. Applying the avoidance principle using that $\mathcal{G}+ct$ is a subsolution of (\ref{MCF rotationally symmetric}), we have that $u_t$ will avoid $\mathcal{G}+ct$ for as long as it exists. 
$\medskip$

Suppose that $\mathcal{G}$ is a distance $C$ from $u_0$ at $x=0$. Then, we have that for $x \in [-\frac{\ell}{2}, \frac{\ell}{2}]$, $|\partial_x u_t| \leq C(t)$ for some constant $C(t)$ depending only on $t$, $C$, and $\ell$. Indeed, since $\partial^2_x u_t \geq 0$, $\partial_x u_t $ cannot be too large at $x=\frac{\ell}{2}$ since that would imply $u_t$ would intersect $\mathcal{G}+ct$, which is a contradiction. Here, the dependence on $C$ and $\ell$ is irrelevant and such choices can be fixed from the outset.
$\medskip$

We may use the time-dependent bound on $|\partial_x u_t|$ over $[-\frac{\ell}{2}, \frac{\ell}{2}]$ in combination with Theorem \ref{EH estimate} to bound $|A|^2$ for $M^1_t$ around the tip, which we identify with $u_t(0)$. Identify the $x$-$y$ plane that $u_t$ is in with the plane $(x,y,0,\dots, 0) \in \mathbb{R}^{n+1}$.  Consider the hyperplane $\mathcal{H}$ perpendicular to the normal $\nu(0) = (0,1,0,\dots,0)$ such that $\mathcal{H}$ contains $y=0$. The normal $\nu(0)$ is the outward normal to $M^1_t$ at the point identified with the tip $u_t(0)$ in the $x$-$y$ plane. Consider a ball $B(0,\text{min}(\frac{\ell}{2},\frac{1}{2}))$ of radius $\text{min}(\frac{\ell}{2},\frac{1}{2})$ in $\mathcal{H}$ centered around $0$. Since $M^1$ is normalized to radius $1$, this means $u_t>1$ for all time and so $M^1_t$ will remain graphical over $B(0,\text{min}(\frac{\ell}{2},\frac{1}{2}))$ for all time. Here we use the fact that $u_t >w$ for all time that it exists, so the flow $M^1_t$ will not collapse onto the axis of rotation. 
$\medskip$

Now, consider the quantity $v = \langle \nu_{M^1_t}, \nu(0)\rangle^{-1}$, where $\nu(0)$ is the unit normal (independent of time) to $M^1_t$ corresponding to the normal to $u_t(0)$. By the bound $|\partial_x u_t|< C(t)$ over $[-\frac{\ell}{2}, \frac{\ell}{2}]$, there is a bound on $v$ depending on $t$ in $B(0,\text{min}(\frac{\ell}{2},\frac{1}{2}))$. Using the bound on $v$, we may apply Theorem \ref{EH estimate} with $R=\text{min}(\frac{\ell}{2},\frac{1}{2})$, the hyperplane $\mathcal{H}$, and some $\theta <1$. This gives a bound $|A|^2 < C(t)$ for the part of $M^1_t$ that is graphical over $B(0,\theta\, \text{min}(\frac{l}{2},\frac{1}{2}))$.
$\medskip$

With this in hand, we move on to finding a bound on $|A|^2$ for the rest of $M^1_t$. By the fact that there is a bound on $|A|^2$ around the tip of $u_t$, we have that there is a bound depending on $t$ for the speed of the tip. That is, $\big|\partial_t u_t (0)\big| \leq C(t)$.
$\medskip$

By symmetry of $u_t$ about $x=0$, we can just consider one side, so let us consider $u_t$ over $x>0$. Consider the unit vector $\nu_{\theta} = (-\sin(\theta), \cos(\theta))$ in the $x$-$y$ plane which we identify with $(-\sin(\theta), \cos(\theta), 0, \dots, 0) \in \mathbb{R}^{n+1}$. Consider $\theta>0$ small but fixed. Let $\mathcal{H}_{\theta}$ be the hyperplane that is orthogonal to $\nu_{\theta}$. Let $\mathcal{L} \subset \mathcal{H}_{\theta}$ be the projection of $u_t|_{\{x\geq \text{min}(\frac{\ell}{4},\frac{1}{4})\}}$ onto $\mathcal{H}_{\theta}$, where the last $n-1$ coordinates of $\mathcal{L}$ are zero. Then, we may find a fixed small enough $R>0$ such that for each $x \in \mathcal{L}$, $M^1_t$ will be a graph\footnote{Technically, $M^1_t$ will be a double-sheeted graph over $B(x,R)$ by considering the part of $u_t|_{x<0}$ lying over $B(x,R)$, but this does not affect the application of Theorem \ref{EH estimate}, which will only be applied to the part of $M^1_t$ corresponding to $u_t|_{x>0}$.} over the ball $B(x,R)\subset \mathcal{H}_{\theta}$ for a time depending on $t$. This is possible because $u_t>1$ and the tip $u_t(0)$ moves at a speed only depending on $t$, as shown above.
\begin{figure}
\centering
\includegraphics[scale = .6]{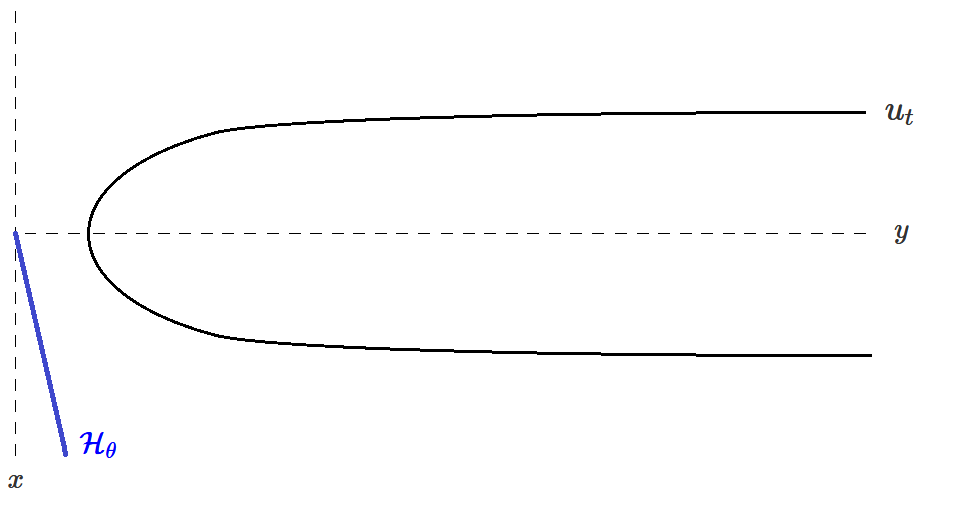}
\caption{An approximate profile of $u_t$ along with the tilted support plane $\mathcal{H}_\theta$. Note that in reality, $\mathcal{H}_\theta$ is not rotationally symmetric about the $x$-axis.}
\end{figure}
$\medskip$

Moreover, the quantity $\langle \nu_{M^1_t}, \nu_{\theta}\rangle^{-1}$, which is just $v$ with respect to $\nu_{\theta}$, is uniformly bounded over such $B(x,R)$. This is because $\mathcal{H}_{\theta}$ is at an angle $-\theta$ with respect to the unit normal $\nu(0)$ at the tip $u_t(0)$, so $\partial_x u_t$ will grow at a linear rate with respect to $\mathcal{H}_{\theta}$ as $u_t$ approaches its asymptote. Thus, applying Theorem \ref{EH estimate} to $B(x,L) \subset \mathcal{H}_{\theta}$ using $R$, the bound on $v$, and an appropriate choice of $\theta$ (given the choice of $\theta$ for the previous application of Theorem \ref{EH estimate}), we get a bound on $|A|^2$ for $u_t$ over $\mathcal{L}$. 
$\medskip$

Putting all of this together, there is a bound $|A|^2 < C(t)$ for $M^1_t$. Since $|A|^2$ is bounded for any $t>0$, this gives that $M^1_t$ will exist for $t \in [0,\infty)$ by applying the short time existence theorem at any finite time to extend the flow. So, $M^1_t$ is an eternal solution as claimed. 
\end{proof}

To understand the asymptotics of the flow, we will need to relate the motion of the profile curve to the curve shortening flow, which will require the following lemma.

\begin{lem}\label{far} Let $M^1_t$ be the eternal solution of Proposition \ref{longtime existence}. As $t \to \infty$, $M^1_t$ becomes infinitely far from its axis of rotation.
\end{lem}
\begin{proof}
Equivalently, we will prove that $u_t$ will become infinitely far from the $y=0$ axis. Suppose that this is not the case. Since $u_t(0)$ is the unique minimum of the flow, we have that $\lim_{t \to \infty} u_t(0)  = C < \infty$. This implies that $\lim_{t \to \infty} \partial_t u_t(0) = 0$. 
$\medskip$

If $|A|^2$ is uniformly bounded for $M^1_t$ as $t \to \infty$, then $M^1_t$ converges to a smooth minimal surface $N$ as $t \to \infty$. However, $N$ must be given by the rotation of a graph symmetric about $x=0$. So, $N$ is a catenoid symmetric about $x=0$, but this must intersect $M^1$, which is a contradiction. 
$\medskip$

Now, suppose that $|A|^2$ is not uniformly bounded as $t \to \infty$. Let $\mathcal{H}_{\theta}$ be the tilted plane as above. Let $\mathcal{L}$ be the projection of $\{x>0\}\cap \{y\geq C\}$ onto $\mathcal{H}_{\theta}$. Here, we are technically taking the part of $\mathcal{L}$ with the last $n-1$ coordinates zero. Since the tip of $u_t$ is stationary as $t \to \infty$, i.e. $u_t(0) \to C$, we have that $M^1_t$ is a graph over small balls centered on $\mathcal{L} \subset \mathcal{H}_{\theta}$ with controlled derivative for all time $t$. We may now apply Theorem \ref{EH estimate} to $M^1_t$ over $\mathcal{L}$. We find that for any distance $d>0$, the set of points on $M^1_t$ which is distance greater than $d$ from the tip $u_t(0)$ has uniformly bounded $|A|^2$, depending on $d$, for all $t \in [0,\infty)$.
$\medskip$

Since we are supposing $|A|^2 \to \infty$ as $t \to \infty$, this leaves the possibility that the curvature is blowing up as $t \to \infty$ near the tip $u_t(0)$. In other words, we may find times $t_i$ and points $q_i \in M^1_{t_i}$ such that $|A|^2(q_i) \to \infty$, where $q_i$ are all within a uniform distance in $\mathbb{R}^{n+1}$ from the tip point $(0,u_t(0), 0,\dots, 0)$. By reflection symmetry across $x=0$, these points come in pairs $q_i, \overline{q_i}$. Note that such points $q_i$ are uniformly bounded away, by mean convexity of the flow, from the axis of rotation and hence we must have $H(q_i) \to \infty$. In particular, $H(q_i) > 2$ for $i$ sufficiently large.
$\medskip$

Because $\lim_{t\to \infty} u_t(0) = C < \infty$ this implies that $H(u_t(0)) \to 0$. Thus, $M^1_t$ must achieve an interior minimum of $H$ at some point $y_t$ on the graph of $u_t$ between $q_i$ and $\overline{q_i}$ for all $i$ large enough. It is easy to see something even stronger in fact:  one can actually find $a(t)$, $\overline{a}(t)$, in lieu of the discrete $q_i$, such that $\liminf_{t \to \infty}|A|^2(a(t))>2$. Then, since $H(u_t(0)) \to 0$, there is $y_t$ on the graph of $u_t$ between $a(t)$ and $\overline{a}(t)$ such that for all sufficiently large times $t$, $M^1_t$ achieves an interior minimum of $H$ at $y_t$. To see the existence of such $a(t)$ for sufficiently large times, if $a(t), \overline{a}(t)$ did not exist, we would find a sequence of times where $M^1_t$ has uniformly bounded curvature, which would converge to a minimal catenoid, as before. This is a contradiction as such a catenoid must intersect $M^1$ yet must also be distinct from $M^1$.
$\medskip$

Using that the tip $u_t(0)$ lies between $a(t)$ and $\overline{a}(t)$ on the graph of $u_t$, we see it has $H$ lower bounded by $H(y_t)$. The $y_t$, by mean convexity and the strict maximum principle, must have $H(y_t) > 0$ and $H(y_t)$ must be increasing for all time, as long as $H(y_t)<2$. This implies that $H(u_t(0))$ must be bounded away from zero for all time, contradicting the fact that $H(u_t(0)) \to 0$. Thus, $M^1_t$ must become infinitely far from its axis of rotation.
\end{proof}

Since the profile curve moves infinitely far from the origin, it behaves like the curve shortening flow. To understand the asymptotics, we will need to take a pointed limit of the profile curve $u_t$ as $t \to \infty$, but this requires uniform curvature bounds, which we find in the following lemma. 

\begin{lem}
Let $M^1_t$ be the eternal solution of Proposition \ref{longtime existence}. Then, there exists $C < \infty$ such that $|A|^2 < C$ for $t \in (-\infty, \infty)$.
\end{lem}

\begin{proof}
Suppose not. Then, we may find a sequence of times $t_m \to \infty$ such that $M^1_{t_m}$ achieves the supremum
\begin{equation}\lambda_m:= \sup_{t \leq t_m} \sup_{M^1_t} |A|
\end{equation}
and $\lambda_m \to \infty$. 
$\medskip$

\noindent We first pick $x_m \geq 0$ such that $u_t(x_m)$ realizes the supremum $\lambda_m$. It is possible to pick all $x_m$ positive by the symmetry of $u_t$ about $x=0$. Then, define 
\begin{equation}\label{rescalings1}
    u^m_t(x) := \lambda_m \big(u_{t/\lambda_m^2 +t_m}(x/\lambda_m + x_m) - u_{t_m}(x_m)\big)
\end{equation}
By (\ref{MCF rotationally symmetric}), we have that $u^m_t$ satisfies the following equation:

\begin{equation}\label{equation the approximate flows satisfy}
    \frac{\partial u^m_t}{\partial t} = \frac{(u^m_t)_{xx}}{1+(u^m_t)_x^2} - \frac{n-1}{\lambda_m u_{t/\lambda_m^2 +t_m}(x/\lambda_m + x_m)}
\end{equation}

For $M^1_t$, $|A|>k$, where $k$ is the curvature of the profile curve $u_t$. Because $u_t$ becomes infinitely far from the axis of rotation and $|A| \to \infty$, $k \nnearrow |A|$ at $x_m$ as $t_m \to \infty$. Then, by definition of $\lambda_m$, we have that $k^m_t \leq 1$ for $t \leq 0$ and $k^m_0(x_m) \nnearrow 1$ as $m \to \infty$. Moreover, the second term in (\ref{equation the approximate flows satisfy}) must approach zero uniformly fast on any compact subset of negative times as $m \to \infty$ by the fact that $u_t$ becomes infinitely far from the axis of rotation. This means that we may pass to a limit to find the nonempty embedded flow $u^{\infty}_t$, defined on $(-\infty, 0]$, which satisfies the equation:
\begin{equation}\label{graphical CSF}
    \frac{\partial u^{\infty}_t}{\partial t} = \frac{(u^{\infty}_t)_{xx}}{1+(u^{\infty}_t)_x^2}
\end{equation}
The equation (\ref{graphical CSF}) is the graphical version of curve shortening flow. Since each $u^m_t$ is a convex curve, $u^{\infty}_t$ is also convex. In short, we have found an embedded convex ancient solution to curve shortening flow which is graphical. In fact, $u_t^{\infty}$ is nonflat since there exists a point $x_0$ on $u_t^{\infty}$ where $k^{\infty}_0(x_0) =1$, which follows from the fact that $k^m_0(x_m) \nnearrow 1$ as $m \to \infty$.
$\medskip$


For $n=2$, $u_t$ is defined over all of $x \in (-\infty, \infty)$, but for $n \geq 3$, $u_t$ will be defined over $x \in (-c,c)$. So, for $n=2$, $u^{\infty}_t$ is defined for $x \in (-\infty, \infty)$. And for $n\geq 3$, since $u^m_t$ is defined for $x \in (-\lambda_m (c+x_m), \lambda_m (c - x_m))$ and since $x_m \geq 0$, we have that $u^{\infty}_t$ is defined for $x \in (-\infty, C)$ for some $0 \leq C \leq \infty$, using that $\lambda_m \to \infty$. By the recent full classification of embedded convex ancient solutions to curve shortening flow due to Bourni-Langford-Tinaglia~\cite{BLT2}, there does not exist an embedded convex nonflat ancient solution to curve shortening flow which is graphical over either a half-infinite interval or all of $\mathbb{R}$. This is a contradiction since we found earlier that $u^{\infty}_t$ is nonflat. Thus, there exists some $C$ such that $|A|^2< C$ for $M^1_t$ for all $t$.
\end{proof}

Since $M^1_t$ has uniformly bounded curvature as $ t \to \infty$, we can now extract a pointed limit of $u_t$ as $t \to \infty$. Indeed, for a sequence of times $t_m \to \infty$, define 
\begin{equation}\label{rescalings2}
    v^m_t(x) := u_{t+t_m}(x) - u_{t_m}(0)
\end{equation}
Using that $u_t$ becomes infinitely far from the axis of rotation as $t \to \infty$, we may extract an embedded convex ancient solution $v^{\infty}_t$ which is graphical, just as in (\ref{equation the approximate flows satisfy}) and (\ref{graphical CSF}). For $n=2$, $v^{\infty}_t$ is defined for $x \in (-\infty, \infty)$ and for $n \geq 3$, $v^{\infty}_t$ is defined for $x \in (-c,c)$. Again, by the classification of Bourni-Langford-Tinaglia~\cite{BLT2}, we have that for $n=2$, $v^{\infty}_t$ is a line and for $n \geq 3$, $v^{\infty}_t$ is a grim reaper with width $c$. Since the choice of limit sequence $t_m$ was arbitrary, we know that these must be the unique limits. Thus, for $n \geq 3$, $u_t$ smoothly converges as $t \to \infty$ to a grim reaper with the same width as that of the catenoid $M^1$ (which follows from the fact that $M^1_t$ must be asymptotic to $M^1$ for all time). And for $n=2$, $u_t$ converges on all compact subsets to a line, in the sense that the curvature $k_t$ approaches zero uniformly as $t \to \infty$.

\section{Proof of Theorem 1.5}
In this section, we prove the partial uniqueness result, Theorem \ref{uniqueness}. Our first step is to understand the topology of such $M_t$, which is simple:
\begin{lem} $M_t$ is an embedded cylinder $S^{n-1} \times \R$.
\end{lem}
\begin{proof} Since $M_t$ is $O(n)\times O(1)$-invariant, it can be denoted by a profile curve $\gamma_t$ in the $x$-$y$ plane, where $x$ is the axis of rotation $\gamma_t$ is reflection symmetric about $y$. We first see it suffices to show $\gamma_t$ stays on one side of the axis of rotation. Indeed, since $M_t$ is connected, $\gamma_t$ is connected and so must be diffeomorphic to either $\R$ or $S^1$. It cannot be $S^1$ because $M_t$ is eternal and hence noncompact. So $\gamma_t$ must be diffeomorphic to $\R$. Thus, if $\gamma_t$ remains on one side of the axis of rotation, $M_t$ is diffeomorphic to an embedded $S^{n-1} \times \R$.
$\medskip$

To see that $\gamma_t$ stays on one side of the axis of rotation, the $x$-axis, we use the reflection symmetry. Suppose $\gamma_t$ crosses the $x$-axis at $x=x_0$. Then since $M_t$ is smooth and embedded, $\gamma_t$ must cross the $x$-axis orthogonally at $x=x_0$. By embeddedness and reflection symmetry across $\{x=0\}$, $x_0 \neq 0$ as $M_t$ may not just be the vertical line $\{x=0\}$. So, $\gamma_t$ must be a subset of either $\{x \leq x_0\}$ or $\{x \geq x_0\}$. This implies that by reflection symmetry and connectedness of $M_t$, we have that $\gamma_t \subset \{x \leq |x_0|\}$. Since $\gamma_t$ crosses the $x$-axis at $x=x_0$, it does so at $x=-x_0$ as well, which implies that $M_t$ is compact. This is a contradiction as $M_t$ must be noncompact since it is an eternal solution. Thus, $\gamma_t$ stays on one side of of the axis of rotation.
\end{proof}

Since the eternal solution $M_t$ is rotationally symmetric and is an embedded cylinder, $M_t$ can be represented as the rotation of a smooth profile curve $\gamma_t$ which lies above the axis of rotation. In the notation established in the lemma, the points above the axis of rotation are the points with $y > 0$, and we may assume without loss of generality that $\gamma_t \subset \{y>0\}$. $M_t$ is also $O(1)$-invariant, so let the $y$-axis be the axis of reflection for $M_t$. Since $M_t$ has uniformly bounded curvature for all time, $\gamma_t$ must remain a uniform distance from the axis of rotation for all time. So, $\gamma_t \subset \{y \geq c>0\}$. In particular, this means that $M_t$ is nonentire, i.e. the flow of $M_t$ does not sweep out all of spacetime. 
$\medskip$

The flow $M_t$ is an embedded cylinder so it divides $\mathbb{R}^{n+1}$ into two connected components: the inside and the outside. The outside is defined to be the component that does not contain the axis of rotation. Let the outward-pointing normal to $M_t$ be the normal that points into the outside component. $M_t$ is assumed to have a sign on mean curvature, so it must be mean convex with respect to either the inward-pointing or the outward-pointing normal. We will deal with these cases separately, which completes the proof of the theorem.

\begin{lem}\label{outward pointing}
Suppose $M_t$ is mean convex with respect to the outward-pointing normal. Then, $M_t$ satisfies the conclusions of Theorem \ref{uniqueness}.
\end{lem}
\begin{proof}
Since $M_t$ has a sign on mean curvature and has uniformly bounded curvature as $t \to -\infty$, it must be smoothly converging to a (a priori empty) union of complete smooth embedded minimal surfaces $N$ as $t \to -\infty$. This follows from the same reasoning as at the beginning of Lemma \ref{l3}. 
$\medskip$

We first see that $N$ is nonempty. Since $M_t$ is mean convex with respect to the outward normal, it bounds a flow of domains $K_t$, where $K_t$ are the outside components of $M_t$, and these nest, meaning that $K_t \subset K_s$ for $t>s$. This means that the closest point in $K_t$ to the axis of rotation will monotonically approach the axis of rotation as $t \to -\infty$. This implies that the limit $N$ must be nonempty, since the $K_t$ nest and monotonically approach the axis as $t \to -\infty$.
$\medskip$

The minimal surfaces comprising $N$ must all be disjoint because if they were not, this would contradict the smoothness or the embeddedness of $M_t$. We also know that each connected component of $N$ must be an $O(n)\times O(1)$-invariant minimal surface since $M_t$ is. In fact, all of the minimal surfaces comprising $N$ must be $O(n) \times O(1)$-invariant about the same axis of rotation and axis of reflection. Thus, $N$ must be given by a union of smooth disjoint profile curves, representing hyperplanes and catenoids, which are rotationally symmetric about the $x$-axis and reflection symmetric about the $y$-axis.
$\medskip$

Since $\gamma_t \subset \{y \geq c\}$, we must have that $N \subset \{y \geq c\}$, where $N$ in this case denotes the profile curves comprising $N$. We can see that the profile curves of the hyperplanes in $N$ must be given by vertical lines $\{x=a\}$. However, since each connected component of $N$ must be a complete surface, we cannot have a complete hyperplane given by $\{x=a\}$ which is also a subset of $\{y \geq c\}$. This means that $N$ consists only of catenoids. Now, none of the catenoids in $N$ may intersect each other because each smooth complete minimal surface in $N$ must be disjoint. Since all catenoids which are rotationally symmetric about the same axis and reflection symmetric about the same axis must transversally intersect each other, we have that if $N$ is nonempty, $N$ must consist of a single catenoid.  Thus, $M_t$ smoothly converges to a catenoid $N$ as $t \to -\infty$. 
$\medskip$

Since $M_t$ bounds the nesting domains $K_t$, $M_t$ must converge to the catenoid $N$ from one side. Since $M_t$ is mean convex with respect to the outward normal, it must converge to the catenoid $N$ from the outside. That is, $M_t$ must be a subset of the outside component of the catenoid (as defined above).
$\medskip$

Suppose that $M_t$ does not become infinitely far from the axis of rotation. Then, again using the uniform bound on curvature and mean convexity, $M_t$ must converge to a nonempty complete smooth minimal surface $N^*$ as $t \to \infty$. The surface $N^*$ must also be rotationally symmetric and reflection symmetric about the same axes as $N$. This implies that $N^*$ is a catenoid or a plane which must transversally intersect $N$. This is a contradiction because $M_t$ must be a subset of the outside component of $N$ for all time. Thus, $M_t$ becomes infinitely far from the axis of rotation.
$\medskip$

Using the uniform curvature bound assumption, we may find an ancient curve shortening flow in the limit as $t \to \infty$, as in the last part of the proof of Theorem \ref{second theorem}. This works even without assuming graphicalness, as it is akin to rescaling (\ref{mean curvature of graph rotation}) with $r \to \infty$. Since $M_t$ is mean convex, the limiting ancient curve shortening flow must be convex since all but one of the principal curvatures on $M_t$ tend to zero as we move away from the axis of rotation by (\ref{mean curvature of graph rotation}). Thus, we find a convex noncompact ancient solution to curve shortening flow in the limit which is disjoint from $\{y < 0\}$ by normalizing the tip to the origin of the $x$-$y$ plane. Using the classification of Bourni-Langford-Tinaglia \cite{BLT}, we find that $M_t$ must converge as $t \to \infty$ to a grim reaper for $n\geq 3$ (since it is trapped in a slab and bounded away from the lower half-plane) and must become flat for $n=2$ (since the flow cannot be contained in any slab). Here, we note that the grim reaper may potentially have a smaller width than that of the catenoid, unlike in Theorem \ref{second theorem}, but the uniform curvature bound prevents the grim reaper from having arbitrarily smaller width than that of the catenoid.  
\end{proof}

\begin{lem}
There is no such $M_t$ which is mean convex with respect to the inward-pointing unit normal. 
\end{lem}
\begin{proof}
Using that $M_t$ has uniformly bounded curvature and is mean convex with respect to the inward-pointing unit normal, we find that $M_t$ converges to a nonempty complete smooth minimal surface $N$ as $t \to \infty$. This follows just as in Lemma \ref{outward pointing}, except the change in orientation of the normal means we must take $t\to \infty$. Similarly, $N$ must be a catenoid and $M_t$ must be approaching $N$ from one side. In fact, $M_t$ must be a subset of the outside of $N$. Moreover, $M_t$ must be becoming infinitely far from the axis of rotation as $t \to -\infty$. Now consider a catenoid $N^*$ of larger radius than $N$ which is rotationally symmetric and reflection symmetric about the same axes as $N$. The catenoid $N^*$ intersects $N$, but for $t$ negative enough, $N^*$ must be disjoint from $M_t$ as $M_t$ becomes infinitely far from the axis of rotation as $t \to -\infty$. This means that $M_t$ may not flow into $N$, which is a contradiction. Thus, $M_t$ does not exist.
\end{proof}

\section{Concluding Remarks}

With the assumption of mean convexity for a closed surface or the assumption of $\alpha$-noncollapsed \cite{HK}, all singularities to mean curvature flow must be modeled on convex ancient solutions. Since our examples of ancient solutions from Theorem \ref{first theorem} are nonconvex, they would not arise as the blowup limits of such flows. To the authors' knowledge there is no actual obstruction to the ancient solutions of Theorem \ref{first theorem} arising as singularity models of the flow of some surface, although such surfaces would certainly be quite rare. In the case the minimal surface $M$ is a catenoid, such a blowup would seem to violate the multiplicity one conjecture. This is because at the singular point (not rescaling) the flow would appear to be two sheets, coming together, joined with a small neck.
$\medskip$

Specific to the reapernoid, the eternal solution of Theorem \ref{second theorem}, blowup limits obtained via rescaling by the supremum of $|A|$ must attain their maximum of $|A|$ at some point in spacetime. It is likely that the reapernoid does not attain its maximum of $|A|$, meaning that it does not arise as such a blowup limit. We do not prove this, as this would require finer control on the curvature of the flow than we find. It remains an open question if all eternal solutions that arise from this blowup procedure must be translating.
$\medskip$

The periodicity/asymptotic flatness assumptions in Theorem \ref{first theorem} also seems to be purely technical and we believe the statement without these qualifiers is true. In other words, any unstable minimal surface should give rise to an ancient solution. Since the existence of an ancient solution coming out of a minimal surface would imply it is unstable, this could potentially provide a characterization of instability of minimal surfaces of uniformly bounded curvature in $\mathbb{R}^n$, if true. 
$\medskip$

The bounded curvature assumption in Theorem \ref{first theorem} seems more necessary. This assumption gives that the flow of each of the perturbations will exist for a long time provided the approximate solutions are chosen sufficiently close. If the minimal surface in question has globally unbounded curvature it seems reasonable to expect that perturbation will quickly ``fly away'' from points of higher and higher curvature so there would be no ancient solution flowing from such minimal surfaces. On the other hand, if one could skirt around the applications of bounded curvature mentioned above, perhaps one could find a related ancient solution by using initial approximate flows that are non-constant perturbations of the minimal surface $M$ depending on the curvature, unlike the equidistant perturbations used in the proof of Theorem \ref{first theorem}.
$\medskip$

\vspace{-.2in}
The construction of the ancient solutions in Theorem \ref{first theorem} given in this article could possibly be applied in the case of ambient manifolds of positive scalar curvature, by using the more general results in Fischer-Colbrie and Schoen \cite{FCS}. That is, in 3-manifolds of positive scalar curvature the stable minimal surfaces are either diffeomorphic to planes or cylinders. It is not directly clear how to adapt some parts of the proofs however, e.g. if the ambient manifold does not admit a group of translations. There could be more specialized classes of ambient spaces where our construction can be carried out by assuming some symmetry. See Choi-Mantoulidis \cite{ChM} for constructions of ancient solutions out of closed minimal surfaces in non-Euclidean ambient manifolds. Note that their techniques differ significantly from those used in this paper and are of wider interest to more general types of flows.
$\medskip$

The uniqueness statement of Theorem \ref{uniqueness} raises questions about how much this can be generalized. We conjecture that the reapernoid eternal solution obtained in Theorem \ref{second theorem} is unique among all $O(n)\times O(1)$-invariant eternal solutions, possibly without a curvature bound and not just in an asymptotic sense. This seems reasonable to expect because the catenoid has index 1 and so we would not expect more than one distinct ancient solution flowing out of it on one side. We also conjecture that one could drop the assumption of the sign on mean curvature in Theorem \ref{uniqueness} and obtain the same result. Similar uniqueness results for other minimal surfaces could also be possible.

\linespread{1}
\bibliographystyle{alpha}

\end{document}